\documentclass[11pt, a4paper]{article}

\usepackage{amsmath,amsfonts,amsthm,amssymb}
\usepackage{hyperref}
\usepackage[english]{babel}

\usepackage{xcolor} 
\hypersetup{
    colorlinks,
    linkcolor={blue},
    citecolor={red},
    urlcolor={blue}
}

\newtheorem{theorem}{Theorem}

\newtheorem{conjecture}[theorem]{Conjecture}

\newtheorem{problem}[theorem]{Problem}

\begin{document}

\title{Coloring tournaments: from local to global \thanks{The first author was supported by a CIMI research fellowship. The third author was partially supported by the ANR Project STINT under Contract ANR-13-BS02-0007.}}

\author{Ararat Harutyunyan$^a$, Tien-Nam Le$^b$,\\ St\'{e}phan Thomass\'{e}$^b$, and Hehui Wu$^c$\\~\\
\small $^a$Institut de Math\'ematiques de Toulouse \\ \small Universit\'e de Toulouse Paul Sabatier
\\ \small 31062 Toulouse Cedex 09, France\\~\\
\small $^b$Laboratoire d'Informatique du Parall\'elisme \\ \small   UMR 5668 ENS Lyon - CNRS - UCBL - INRIA
\\ \small Universit\'e de Lyon, France\\~\\
\small $^c$Shanghai Center for Mathematical Sciences \\ \small Fudan University 
\\ \small 220 Handan Road, Shanghai, China
}

\date{}

\maketitle
  \begin{abstract}
The \emph{chromatic number} of a directed graph $D$ is the minimum number of colors needed to color the vertices of $D$ such that each color class of $D$ induces an acyclic subdigraph. Thus, the chromatic number of a tournament $T$ is the minimum number of transitive subtournaments which cover the vertex set of $T$. We show in this paper that tournaments are significantly simpler than graphs with respect to coloring. Indeed, while undirected graphs can be altogether ``locally simple" (every neighborhood is a stable set) and have large chromatic number, we show that locally simple tournaments are indeed simple. In particular, there is a function $f$ such that if the out-neighborhood of every vertex in a tournament $T$ has chromatic number at most $c$, then $T$ has chromatic number at most $f(c)$. This answers a question of Berger et al.
  \end{abstract}


\section{Introduction}

A directed graph is said to be \emph{acyclic} if it does not contain any directed cycles.
Given a loopless digraph $D$, a \emph{$k$-coloring} of $D$ is a coloring of each of the
vertices of $D$ with one of the colors from the set $\{1,...,k\}$ such that
each color class induces an acyclic subdigraph. The \emph{chromatic number} $\vec \chi(D)$ of $D$ is the smallest
number $k$ for which $D$ admits a $k$-coloring. This digraph invariant was introduced by Neumann-Lara \cite{Neu82}, and 
naturally generalizes many results on the graph chromatic number (see, for example, \cite{BFJKM04}, \cite{HM11}
\cite{HM11b}, \cite{HM12}, \cite{KLMR13}). In this paper, we study the chromatic number of a class of tournaments
where the out-neighborhood of every vertex has bounded chromatic number.

A \emph{tournament} is a loopless digraph such that for every pair of 
distinct vertices $u,v$, exactly one of $uv, vu$ is an arc. 
Given a tournament $T$, a subset $X$ of $V(T)$ is \emph{transitive} if the subtournament of 
$T$ induced by $X$ contains no directed cycle.
Thus, $\vec \chi (T)$ is the minimum $k$ such that $V(T)$ can be colored with $k$ colors where
each color class is a transitive set. 
The coloring of tournaments has close relationship with the celebrated Erd\H{o}s--Hajnal conjecture (cf. \cite{APS01, EH89}) and
has been studied in \cite{heroes,CCS14,C14,BCC15,CKLST}.

Given $t\ge 1$, a tournament $T$ is {\it $t$-local} if for every vertex
$v$, the subtournament of $T$ induced by the set of out-neighbors of $v$ has chromatic number at most $t$.
The following conjecture was raised in \cite{heroes} (Conjecture 2.6)
and settled for $t=2$ in \cite{CKLST}.

\begin{conjecture}\label{conj}
There is a function $f$ such that every $t$-local tournament $T$ satisfies $\vec \chi (T)\leq f(t)$.
\end{conjecture}
The goal of this note is to provide a proof of Conjecture \ref{conj} for all $t$. 

Given a set $S \subset V(T)$, we say
that $S$ is a \emph{dominating set} of $T$ if 
every vertex in $V \setminus S$ has an in-neighbor in $S$. The \emph{dominating number} $\gamma(T)$ of a tournament $T$ is the smallest number $k$ such that $T$ has a dominating set of size $k$.
The main tool to prove Conjecture \ref{conj} is the following theorem, which seems more interesting than our original goal.

\begin{theorem}\label{dom-chrom}
For every integer $k\ge 1$, there exist integers $K$ and $\ell$ such that every tournament $T$ with dominating number at least $K$ contains a subtournament on $\ell$ vertices and chromatic number at least $k$.
\end{theorem}

Roughly speaking, Theorem \ref{dom-chrom} asserts that if the dominating number of a tournament is sufficiently large, then it contains a bounded-size subtournament with large chromatic number. One may ask whether high dominating number is enough to force an induced copy of a specific (high chromatic number) subtournament. The following tournaments may be potential candidates.
Let $S_1$ be the tournament with a single vertex. For every $i>1$, let $S_i$ be the tournament (with $2^i-1$ vertices) obtained by blowing up two vertices of an oriented triangle into two copies of $S_{i-1}$. It is easy to check that $\vec \chi(S_i)\ge i$. The following problem is trivial for $i\le 2$ and verified for $i=3$ in \cite{CKLST}, while still open for all $i\ge 4$.

\begin{problem}\label{clique}
For every integer $i\ge 1$, there exist $f(i)$ such that every tournament $T$ with dominating number at least $f(i)$ contains an isomorphic copy of $S_i$.
\end{problem}

On another note, it is natural to ask whether Theorem \ref{dom-chrom} still holds with a weaker hypothesis.
In particular, is it true that for every $k$, if the chromatic number of a tournament is huge, then it contains a bounded-size subtournament with chromatic number at least $k$? Unfortunately, the answer is negative for any $k\ge 3$. It is well-known that for any $\ell$, there is an undirected simple graph $G$ with arbitrarily high chromatic number and girth at least $\ell+1$. We fix an arbitrary enumeration of vertices of $G$ and create a tournament $T$ as follows: If $ij$ with $i<j$ is an edge of $G$ then $ij$ is an arc of $T$; otherwise, $ji$ is an arc of $T$. 
Then $T$ has arbitrarily high chromatic number while every subtournament of $T$ of size $\ell$ has chromatic number at most $2$.
However, a similar question for dominating number is still open.

\begin{problem}\label{dom-dom}
For every integer $k\ge 1$, there exist integers $K$ and $\ell$ such that every tournament $T$ with dominating number at least $K$ contains a subtournament with $\ell$ vertices and dominating number at least $k$.
\end{problem}

\section{Proof of Conjecture \ref{conj}}

For every vertex $v$ in a tournament $T$, we denote by $N^+_T(v)$
the set of out-neighbors of $v$ in $T$. Given a subset $X$ of $V(T)$, 
let $N^+_T(X)$ denote the union of all $N^+_T(v)$, for $v\in X$, and 
denote by $N^+_T[X]:=X\cup N^+_T(X)$.
For every subset $X$ of $V(T)$, let $\vec \chi _T(X) $ denote the chromatic number of 
the subtournament of $T$ induced by $X$. 

Given a tournament $T$ and a subset $X$ of $V(T)$, we
say a set $R\subseteq V(T)$ (not necessary disjoint from $X$) is a dominating set of $X$ in $T$ if every vertex in $X\backslash R$ has an in-neighbor in $R$. The \emph{dominating number} $\gamma_T(X)$ of $X$ in $T$ is the smallest number $k$ such that $X$ has a dominating set of size $k$.
When it is clear in the context, we omit the subscript $T$ in the notation.

Let $T$ be a tournament and $X,Y\subseteq V(T)$. The following inequalities are straightforward:
\begin{equation}\label{1}
\gamma_T (N^+[X])\le|X|,
\end{equation}
and
\begin{equation}\label{2}
\gamma_T(Y) \le \gamma_T(X)+ \gamma_T(Y\backslash X).
\end{equation}

Let us restate Theorem \ref{dom-chrom}.

\begin{theorem}\label{dom-chrom2}
For every integer $k\ge 1$, there exist integers $K$ and $\ell$ such that every tournament $T$ with $\gamma(T)\ge K$ contains a subtournament $A$ on $\ell$ vertices and $\vec\chi(A)\ge k$.
\end{theorem}

\begin{proof}  We proceed by induction on $k$. 
The claim is trivial for $k=1$. For $k=2$, we can choose $K=2$ and $\ell=3$. Indeed, if a tournament $T$ satisfies 
$\gamma(T)\geq K=2$, then $T$ is not transitive and thus it contains an oriented triangle $A$ of size $\ell=3$
and $\vec \chi (A)\geq k=2$. 

Assuming now that $(K,\ell)$ exists for $k$, we want to 
find $(K',\ell')$ for $k+1$. For this, we set $K':=k(K+\ell+1)+K$, and fix $\ell'$ later. 
Let $T$ be a tournament such that $\gamma(T)\geq K'$.
Let $D$ be a dominating set of $T$ of minimum size.
Consider a subset $W$ of $D$ of size $k(K+\ell +1)$. From (\ref{1}) and (\ref{2}) we have 
$$\gamma (V\setminus N^+[W]) \ge \gamma (T)-\gamma (N^+[W])\ge K'-|W|\ge K,$$
where $V$ is the vertex set of $T$. Thus by induction hypothesis applied to $k$,
one can find a set $A \subseteq V\setminus N^+[W]$ such that $A$ has $\ell$ vertices 
and $\vec \chi (A)\geq k$. Note that by construction, $A\cap W=\emptyset$ and all arcs between $A$ and $W$ are directed
from $A$ to $W$.

Consider now a subset $S$ of $W$ of size $K+\ell +1$. We claim that $\gamma (N^+(S))\geq K+\ell $. 
If not, we can choose a dominating set $S'$ of $N^+(S)$ of size at most $K+\ell -1$. 
Note that $x$ dominates $S$ for any $x\in A$, and so $S'\cup \{x\}$ dominates $N^+[S]$. Hence $(D\setminus S)\cup S'\cup \{x\}$ would be a dominating set of $T$ of size less than $|D|$, which contradicts the minimality of $|D|$.
Therefore $\gamma (N^+(S))\geq K+\ell $. 

Let $N'$ be the set of vertices $N^+(S)\setminus N^+(A)$.  From (\ref{1}) and (\ref{2}) we have 
$$\gamma (N')\geq \gamma (N^+(S))-\gamma(N^+(A))\geq K+\ell -|A|=K.$$
Thus by induction hypothesis applied to $k$, there is a subset $A_S$ of $N'$ such that $|A_S|=\ell$ and $\vec \chi (A_S)\geq k$.  
Note that by construction, $A_S\cap A=\emptyset$ and all arcs between $A_S$ and $A$ are directed
from $A_S$ to $A$.

We now construct our subtournament of $T$ with chromatic number at least $k+1$. 
For this 
we consider the set of vertices 
$A\cup W$ to which we add the collection of $A_S$, 
for all subsets $S\subseteq W$ of size $K+\ell+1$. Call $A'$ 
this new tournament and observe that its number of vertices is at most 

$$\ell':=\ell+k(K+\ell+1)+\ell {k(K+\ell +1) \choose K+\ell+1 }.$$ 
To conclude, it is sufficient to show that 
$\vec \chi (A') \geq k + 1$. Suppose not, and for contradiction, take a $k$-coloring of $A'$. 
Since
$|W|=k(K+\ell+1)$ there is a monochromatic set $S$ in $W$ of size $K+\ell+1$
(say, colored 1). Recall that we have all arcs from $A_S$ to $A$ and all arcs from $A$ to $S$, and 
note that since $\vec \chi(A)\ge k$ and $\vec \chi(A_S)\ge k$, both $A$ and $A_S$ have a vertex of each of the $k$ colors. Hence 
there are $u\in A$ and $w\in A_S$ colored 1. Since $A_S\subseteq N^+(S)$,
there is $v\in S$ such that $vw$ is an arc. We then obtain the monochromatic cycle $uvw$ of color 1, a contradiction. 
Thus, $\vec \chi (A') \geq k+1$, completing the proof.
\end{proof}



We now show that Conjecture \ref{conj} is true. 

\begin{theorem}
There is a function $f$ such that every $t$-local tournament $T$ satisfies $\vec \chi(T)\leq f(t)$.
\end{theorem}
\begin{proof}  

Let $(K,\ell)$ satisfy Theorem \ref{dom-chrom2} for $k:=t+1$.
Let $T$ be a $t$-local tournament.
Thus, if $\gamma(T)\geq K$ then $T$ contains a set $A$ of 
$\ell$ vertices and $\vec \chi(A)\geq t+1$. 
If a vertex $v\in V(T)\backslash A$ does not have an in-neighbor in $A$, then $A\subseteq N^+(v)$, and so $t+1\le \vec \chi(A)\le \vec \chi( N^+(v))\le t$, a contradiction.
Hence, $A$ is a dominating set
of $T$. 
Note that $$\vec\chi (N^{+}[v])\le \vec\chi (N^{+}(v))+\vec\chi (\{v\})\le t+1$$
for every $v\in V(T)$. Thus
$$\vec \chi(T) = \vec \chi (N^{+}[A]) \le \sum_{v\in A}\vec \chi (N^{+}[v]) \leq (t+1)|A|=(t+1)\ell.$$

Otherwise, $\gamma(T)< K$. Let $D$ be a dominating set of $T$ with minimum size. Then 
$$\vec \chi(T) = \vec \chi (N^{+}[D]) \le \sum_{v\in D}\vec \chi (N^{+}[v]) \leq (t+1)|D|<(t+1)K.$$
Consequently, $t$-local tournaments have chromatic number at most $f(t):=\max\big((t+1)K,(t+1)\ell\big)$.
\end{proof}
\medskip 

The implication of our result is that we are possibly missing a key-definition 
of what is a ``large" (or ``dense") hypergraph (i.e., a set of subsets). It could 
be that for a suitable definition of ``large" (for which ``large" intersecting
``large" would be ``large"), we would obtain that for any tournament $T$ on vertex set $V$,
the set of out-neighborhoods of vertices of $T$ is ``large", and in addition 
the set of subsets of vertices of 
a $K$-chromatic tournament inducing at least chromatic number $k$ is also ``large".
Hence, if two large sets are intersecting in a non-empty way, one could find 
an out-neighborhood with chromatic number $k$. 

If such a notion would exist, it should decorrelate the two large sets 
(out-neighborhoods and $k$-chromatic), and thus imply 
the following: If $T_1,T_2$ are tournaments on the same set of vertices and 
$\vec \chi(T_1)$ is huge, then there is a vertex $v$ such that $T_1$ induces on 
$N_{T_2}^+(v)$ a subtournament of large chromatic number. A very similar conjecture was proposed by
Alex Scott and Paul Seymour.

\begin{conjecture}\cite{SS}
For every $k$, there exists $K$ such that if $T$ 
and $G$ are respectively a tournament and a graph on the same set of vertices 
with $G$ of chromatic number at least $K$, then there is a vertex $v$ such that $G$ induces on 
$N_{T}^+(v)$ a subgraph of $G$ of chromatic number at least $k$.
\end{conjecture}

\end{document}